\documentclass[amstex,12pt,russian,amssymb]{article}

\usepackage{mathtext}
\usepackage[cp1251]{inputenc}
\usepackage[T2A]{fontenc}
\usepackage[russian]{babel}
\usepackage[dvips]{graphicx}
\usepackage{amsmath}
\usepackage{amssymb}
\usepackage{amsxtra}
\usepackage{latexsym}
\usepackage{ifthen}

\def\XXint#1#2#3{{\setbox0=\hbox{$#1{#2#3}{\int}$}
\vcenter{\hbox{$#2#3$}}\kern-.5\wd0}}

\begin{document}

\newcounter{lemma}
\newcommand{\lemma}{\par \refstepcounter{lemma}%
{\bf Лемма \arabic{lemma}.}}

\newcounter{corollary}
\newcommand{\corollary}{\par \refstepcounter{corollary}%
{\bf Следствие \arabic{corollary}.}}

\newcounter{remark}
\newcommand{\remark}{\par \refstepcounter{remark}%
{\bf Замечание \arabic{remark}.}}

\newcounter{theorem}
\newcommand{\theorem}{\par \refstepcounter{theorem}%
{\bf Теорема \arabic{theorem}.}}

\newcounter{proposition}
\newcommand{\proposition}{\par \refstepcounter{proposition}%
{\bf Предложение \arabic{proposition}.}}

\renewcommand{\refname}{\centerline{\bf Список литературы}}

\newcommand{\proof}{{\it Доказательство.\,\,}}

\noindent УДК 517.5

{\bf Р.Р.~Салимов} (Институт математики НАН Украины),

{\bf Е.А.~Севостьянов} (Житомирский государственный университет им.\
И.~Франко)

\medskip
{\bf Р.Р.~Салімов} (Інститут математики НАН України),

{\bf Є.О.~Севостьянов} (Житомирський державний університет ім.\
І.~Франко)

\medskip
{\bf R.R.~Salimov} (Institute of Mathematics of NAS of Ukraine),

{\bf E.A.~Sevost'yanov} (Zhitomir State University of I.~Franko)

\medskip
{\bf О некоторых локальных свойствах пространственных обобщённых
квазиизометрий}

\medskip
{\bf Про деякі локальні властивості просторових узагальнених
квазіізометрій}

\medskip
{\bf On some local properties of space generalized quasiisometries}

\medskip
Получена оценка сверху для меры образа шара в некотором классе
отображений, являющихся обобщением пространственных квазиизометрий с
ветвлением. Как следствие, для указанных отображений получен аналог
классической леммы Шварца при некотором дополнительном ограничении
интегрального характера. Полученные результаты имеют соответствующие
приложения в классах Соболева и Орлича--Соболева.

\medskip
Отримано оцінку зверху для міри образу кулі в деякому класі
відображень, що є узагальненням просторових квазіізометрій з
розгалуженням. Як наслідок, для вказаних відображень отримано аналог
класичної леми Шварца за деяким додатковим обмеженням інтегрального
характеру. Отримані результати мають відповідні застосування в
класах Соболєва і Орліча--Соболєва.

\medskip
For some class of mappings, which are generalization of space
quasiisometries,  an upper estimate for a measure of image of a ball
is obtained. As consequence, it is obtained one analog of Schwartz
lemma for mappings mentioned above. Results of the paper are
applicable in Sobolev and Orlicz--Sobolev classes.

\newpage
{\bf 1. Введение.} Здесь и далее $D$ -- область в ${\Bbb R}^n,$
$n\geqslant 2,$ $m$ -- мера Лебега в ${\Bbb R}^n,$ отображение
$f:D\rightarrow {\Bbb R}^n$ есть непрерывное соответствие
$f(x)=(f_1(x),\ldots, f_n(x)),$ где $x=(x_1,\ldots, x_n).$ Напомним,
что {\it кривой} $\gamma$ мы называем непрерывное отображение
отрезка $[a,b]$ (открытого, либо полуоткрытого интервала одного из
видов: $(a,b),$ $[a,b),$ $(a,b],$ ) в ${\Bbb R}^n,$
$\gamma:[a,b]\rightarrow {\Bbb R}^n.$ Под семейством кривых $\Gamma$
подразумевается некоторый фиксированный набор кривых $\gamma,$ а
$f(\Gamma)=\left\{f\circ\gamma|\gamma\in\Gamma\right\}.$ Следующие
определения могут быть найдены, напр., в \cite[разд.~1--6,
гл.~I]{Va}. Всюду далее
$$A(r_1,r_2,x_0)=\{ x\,\in\,{\Bbb R}^n :
r_1<|x-x_0|<r_2\}\,, S(x_0,r)=\{x\in {\Bbb R}^n: |x-x_0|=\}\,,$$
$$B(x_0, r)=\left\{x\in{\Bbb R}^n: |x-x_0|< r\right\},\quad {\Bbb
B}^n := B(0, 1)\,, \quad {\Bbb S}^{n-1}:=S(0, 1)\,,$$
$\Omega_n$ -- объём единичного шара ${\Bbb B}^n$ в ${\Bbb R}^n,$ а
$\omega_{n-1}$ --  площадь единичной сферы ${\Bbb S}^{n-1}$ в ${\Bbb
R}^n.$ Для произвольных множеств $E,$ $F\subset \overline{{\Bbb
R}^n}:={\Bbb R}^n\cup\{\infty\}$ через $\Gamma(E,F,D)$ в дальнейшем
обозначается семейство всевозможных кривых
$\gamma:[a,b]\rightarrow\overline{{\Bbb R}^n},$ соединяющих $E$ и
$F$ в $D$ (т.е., $\gamma(a)\in E,$ $\gamma(b)\in F$ и $\gamma(t)\in
D$ при $t\in (a, b)$). Борелева функция $\rho:{\Bbb
R}^n\,\rightarrow [0,\infty]$ называется {\it допустимой} для
семейства $\Gamma$ кривых $\gamma$ в ${\Bbb R}^n,$ если
$\int\limits_{\gamma}\rho (x)\ \ |dx|\geqslant 1$
%
%
для всех кривых $ \gamma \in \Gamma.$ В этом случае мы пишем: $\rho
\in {\rm adm} \,\Gamma .$
Пусть $p\geqslant 1,$ тогда {\it $p$ -- модулем} семейства кривых
$\Gamma $ называется величина
\begin{equation}\label{eq2**} M_p(\Gamma)=\inf_{\rho \in \,{\rm adm}\,\Gamma} \int\limits_D
\rho ^p (x)\ \ dm(x)\,.
\end{equation}
Как известно, в основу геометрического определения квазиконформных
отображений $f:D\rightarrow {\Bbb R}^n,$ $n\geqslant 2,$ может быть
положено неравенство
\begin{equation}\label{eq1*!}
M_n(f(\Gamma))\leqslant K\,M_n(\Gamma)\,,
\end{equation}
которое должно выполняться для произвольного семейства $\Gamma$
кривых $\gamma$ в области $D$ (см., напр.,  \cite[разд.~13,
гл.~2]{Va}). Пусть теперь $x_0\in D,$ $r_0\,=\,{\rm dist}\,
(x_0,\partial D),$ $Q:D\rightarrow\,[0,\infty]$ -- некоторая
заданная измеримая по Лебегу функция. Обозначим через $S_i:=S(x_0,
r_i),$ где $0<r_1<r_2<\infty.$ Предположим, что отображение $f$
вместо соотношения (\ref{eq1*!}) удовлетворяет условию
\begin{equation}\label{eq3*!}
M_p\left(f\left(\Gamma\left(S_1,\,S_2,\,A\right)\right)\right)\leqslant
\int\limits_{A} Q(x)\cdot \eta^p(|x-x_0|)\, dm(x)\end{equation}
для произвольной измеримой по Лебегу функции $\eta:
(r_1,r_2)\rightarrow [0,\infty ]$ такой, что
\begin{equation}\label{eqA2}
\int\limits_{r_1}^{r_2}\eta(r) dr\geqslant 1\,,
\end{equation}
где $A=A(r_1,r_2, x_0),$ $0<r_1<r_2< r_0={\rm dist}\,(x_0,
\partial D)$ -- сферическое кольцо с центром в точке $x_0$ радиусов $r_1$ и $r_2.$
Тогда будем говорить, что $f$ является {\it кольцевым $(p,
Q)$-отображением в точке $x_0\,\in\,D$}.  Отметим, что неравенство
(\ref{eq3*!}) соответствует условию (\ref{eq1*!}) при $p=n$ и
$Q(x)\leqslant K,$ а при $Q\leqslant K$ и произвольном $p$ --
неравенству
\begin{equation}\label{eq1**} M_p(f(\Gamma))\leqslant K\,M_p(\Gamma)\,,
\end{equation}
выполненному для произвольного семейства $\Gamma$ кривых $\gamma$ в
области $D.$

\medskip
Целью настоящей заметки является установление некоторых локальных
свойств отображений, удовлетворяющих соотношению (\ref{eq3*!}) и
имеющих важные приложения к классам Соболева и Орлича--Соболева (см.
\cite{SalSev$_1$} и \cite{KRSS}). Здесь нас интересует главным
образом случай неограниченных $Q,$ поскольку соответствующие
исследования <<ограниченного>> случая были проведены другим авторами
ранее (см. \cite{Ge} и \cite{I}).

\medskip В частности, при дополнительном предположении, что $f$ в (\ref{eq1**})
является гомеоморфизмом и $p\in (n-1, n)$ известным математиком Ф.
Герингом установлено, что $f$ является {\it локально
квазиизометричным}, другими словами, при некоторой постоянной $C>0$
и всех $x_0\in D$ справедлива оценка
\begin{equation}\label{eq13}
\limsup\limits_{x\rightarrow
x_0}\frac{|f(x)-f(x_0)|}{|x-x_0|}\leqslant C\,,\end{equation}
см., напр., \cite[теорема~2]{Ge}. При $p=n$ свойство
квазиизометричности указанных отображений, к сожалению,
утрачивается, однако, японским математиком К.~Икома в этом случае
был получен следующий аналог леммы Шварца для аналитических функций:
предположим, что $f:{\Bbb B}^3\rightarrow{\Bbb B}^3$ квазиконформное
отображение, удовлетворяющее условию $f(0)=0,$ преобразующее каждый
радиус единичного шара в кривую, ортогональную к образу сферы
$|x|=r$ при всех $r>0,$ $r<1.$ Тогда
\begin{equation}\label{eqA1}
\liminf\limits_{x\rightarrow 0}\frac{|f(x)|}{\quad
|x|^{\left(\frac{1}{K}\right)^{1/2}}}\leqslant 1\,,
\end{equation}
где $K$ -- постоянная квазиконформности, определяемая из неравенств
$(1/K)\, M(\Gamma)\leqslant M(f(\Gamma))\leqslant K\,M(\Gamma)$
(см., напр., \cite[теорема~2]{I}). В настоящей работе нами
устанавливается справедливость следующего результата.

\medskip
\begin{theorem}\label{th1}
Пусть $n\geqslant 2,$ $1<p<n,$ $f:{\Bbb B}^n\rightarrow {\Bbb B}^n$
-- открытое дискретное кольцевое $(p, Q)$-отоб\-ра\-же\-ние и
$f(0)=0.$ Предположим, что $Q:{\Bbb B}^n\rightarrow[0,\,\infty]$ --
локально интегрируемая функция в ${\Bbb B}^n,$ удовлетворяющая
условию
$$Q_0=\liminf\limits_{\varepsilon\rightarrow 0}\frac{1}{\Omega_n\cdot\varepsilon^n}
\int\limits_{B(0, \varepsilon)} Q(x)dm(x)<\infty\,.
$$
Тогда имеет место оценка:
\begin{equation}\label{eqks1.21}
\liminf\limits_{x\rightarrow 0}\frac{|f(x)|}{|x|}\leqslant c_{0}
Q_0^{\frac{1}{n-p}}\,,
\end{equation}
где  $c_0$  -- некоторая положительная постоянная, зависящая только
от размерности пространства $n$ и $p.$
\end{theorem}

{\bf 2. Вспомогательные сведения.} Следуя работе
\cite[раздел~5]{MRV$_1$}, пару $E=(A,C),$ где $A\subset{\Bbb R}^n$
-- открытое множество и $C$ -- непустое компактное множество,
содержащееся в $A,$ называем {\it конденсатором} в ${\Bbb R}^n.$
Говорят также, что конденсатор $E=(A,C)$ лежит в области $D,$ если
$A\subset D.$ Очевидно, что если $f:D\rightarrow{\Bbb R}^n$ --
непрерывное открытое отображение и $E=(A,C)$ -- конденсатор в $D,$
то $f(E):=(f(A),f(C))$ также является конденсатором в $f(D).$

\medskip
Говорят, что функция $u:A\rightarrow {\Bbb R}$ {\it абсолютно
непрерывна на прямой}, имеющей непустое пересечение с $A,$ если она
абсолютно непрерывна на любом отрезке этой прямой, заключенном в
$A.$ Функция $u:A\rightarrow {\Bbb R}$ принадлежит классу $ACL,$
если она абсолютно непрерывна на почти всех прямых, параллельных
любой координатной оси. Обозначим через $C_0(A)$  множество всех
непрерывных функций $u:A\rightarrow{\Bbb R}^1$ с компактным
носителем, $W_0(E)=W_0(A,C)$ -- семейство неотрицательных функций
$u:A\rightarrow{\Bbb R}^1$ таких, что 1) $u\in C_0(A),$ 2)
$u(x)\geqslant1$ для $x\in C$ и 3) $u$ принадлежит классу $ACL.$
Также обозначим
$ |\nabla u|={\left(\sum\limits_{i=1}^n\,{\left(\frac{\partial
u}{\partial x_i}\right)}^2 \right)}^{1/2}.$
При $p\geqslant 1$ величину
$${\rm cap}_p\,E={\rm cap}_p\,(A,C)=\inf\limits_{u\in W_0(E)}\,
\int\limits_{A}\,|\nabla u|^p\,dm(x)$$
называют {\it $p$-ёмкостью} конденсатора $E.$ Ёмкости в контексте
теории отображений достаточно хорошо отражены в монографии
\cite{GR}. Известно, что при $p\geqslant 1$
\begin{equation}\label{eqks2.8} {\rm
cap}_p\,E\geqslant\frac{\left(\inf
m_{n-1}\,\sigma\right)^{p}}{\left[m(A\setminus
C)\right]^{p-1}},\end{equation}
где $m_{n-1}\,\sigma$ -- $(n-1)$-мерная мера Лебега
$C^{\infty}$-многообразия $\sigma,$ являющегося границей
$\sigma=\partial U$ ограниченного открытого множества $U,$
содержащего $C$ и содержащегося вместе со своим замыканием
$\overline{U}$ в $A,$ а точная нижняя грань берется по всем таким
$\sigma,$ см. \cite[предложение~5]{Kru}. При $1<p<n$
\begin{equation}\label{maz} {\rm cap}_p\,E\geqslant
n{\Omega_n}^{\frac{p}{n}}
\left(\frac{n-p}{p-1}\right)^{p-1}\left[m(C)\right]^{\frac{n-p}{n}}\end{equation}
где ${\Omega}_n$ -  объем  единичного шара  в ${\Bbb R}^n,$ см.,
напр., \cite[неравенство~(8.9)]{Maz}.

\medskip
Следующее утверждение имеет важное значение для доказательства
дальнейших результатов (см.~\cite[предложение~10.2 и замечание~10.8,
гл.~II]{Ri}).

\medskip
\begin{proposition}\label{pr1*!}Пусть
$E=(A,\,C)$ --- произвольный конденсатор в ${\Bbb R}^n,$ и
$\Gamma_E$ --- семейство всех кривых вида $\gamma\colon
[a,\,b)\rightarrow A,$ таких что $\gamma(a)\in C$ и
$|\gamma|\cap\left(A\setminus F\right)\ne\varnothing$ для
произвольного компакта $F\subset A.$
Тогда для произвольного $p\geqslant 1$ имеет место равенство:
\begin{equation}\label{eq1.4.1}
{\rm cap}_p\,E=M_p(\Gamma_E)\,.
\end{equation}
\end{proposition}

\medskip
Пусть $D$ -- область в ${\Bbb R}^n,$ $n\geqslant 2,$ и
$Q:D\rightarrow[0,\infty]$ -- измеримая по Лебегу функция. Тогда
положим
$$q_{x_0}(r)=
\frac{1}{\omega_{n-1}r^{n-1}}\int\limits_{S(x_0,\,r)}Q(x)d\mathcal{H}^{n-1}\,,$$
где $\mathcal{H}^{n-1}$ -- $(n-1)$-мерная мера Хаусдорфа. Следующая
лемма позволяет установить для отображения $f$ выполнение свойств
(\ref{eq3*!})--(\ref{eqA2}) в точке $x_0$ без обременительной
проверки бесконечного числа неравенств в (\ref{eq3*!}) (см. также
\cite[лемма~1]{SalSev}). Ниже мы придерживаемся следующих
стандартных соглашений: $a/\infty = 0$ для $a\ne\infty,$ $a/0=\infty
$ для $a>0$ и $0\cdot\infty =0.$

\medskip
\begin{lemma}\label{lem1} Пусть $n\geqslant 2,$ $p\geqslant 1,$ $Q:D\rightarrow [0, \infty]$ --
заданная измеримая по Лебегу функция, $f:D\rightarrow
\overline{{\Bbb R}^n}$ -- открытое дискретное кольцевое $(p,
Q)$-отображение в точке $x_0\in D$ и $E$ -- конденсатор вида
$E=\left(B(x_0, r_2), \overline{B(x_0, r_1)}\right),$ $0<r_1<r_2<
{\rm dist} \, (x_0,\partial D).$ Полагаем
\begin{equation}\label{eq9}
I\ =\ I(x_0,r_1,r_2)\ =\ \int\limits_{r_1}^{r_2}\
\frac{dr}{r^{\frac{n-1}{p-1}}q_{x_0}^{\frac{1}{p-1}}(r)}\,.
\end{equation}
Тогда для конденсатора $f(E)=\left(f\left(B(x_0, r_2)\right),
f\left(\overline{B(x_0, r_1)}\right)\right)$ выполнено соотношение
\begin{equation}\label{eq2A}
{\rm cap}_p\, f(E)\leqslant\frac{\omega_{n-1}}{I^{p-1}}\,.
\end{equation}
\end{lemma}
\begin{proof} Заметим, что пара $f(E)=\left(f\left(B(x_0, r_2)\right),
f\left(\overline{B(x_0, r_1)}\right)\right)$ является конденсатором
в ${\Bbb R}^n,$ ибо $f$ открыто и непрерывно в $D,$ следовательно,
множество $f\left(\overline{B(x_0, r_1)}\right)$ является компактным
подмножеством $f\left(B(x_0, r_2)\right).$ Не ограничивая общности
рассуждений, можно полагать, что $I \ne 0,$ так как  в противном
случае соотношение (\ref{eq2A}) очевидно, выполнено. Можно также
считать, что $I\ne \infty,$ так как, в противном случае, в
соотношении (\ref{eq2A}) можно рассмотреть $Q(x)+\delta$ (со сколь
угодно малым $\delta$) вместо $Q(x),$ а затем перейти к пределу при
$\delta\rightarrow 0.$ Пусть $I\ne\infty.$ Тогда $q_{x_0}(r)\ne 0$
п.в. на $(r_1,r_2).$ Полагаем
$\psi(t)= \left \{\begin{array}{rr}
1/[t^{\frac{n-1}{p-1}}q_{x_0}^{\frac{1}{p-1}}(t)]\ , & \ t\in
(r_1,r_2)\ ,
\\ 0\ ,  &  \ t\notin (r_1,r_2)\ .
\end{array} \right. $
Тогда виду теоремы Фубини
\begin{equation}\label{eq3}
\int\limits_{A} Q(x)\cdot\psi^p(|x-x_0|) dm(x)=\omega_{n-1} I\,,
\end{equation}
где, как и прежде, $A=A(r_1,r_2, x_0)=\{ x\,\in\,{\Bbb R}^n :
r_1<|x-x_0|<r_2\}.$ Заметим, что функция $ \eta_1(t)= \psi(t)/I,$
$t\in (r_1,r_2),$ удовлетворяет соотношению вида (\ref{eq3*!}),
поскольку $\int\limits_{r_1}^{r_2}\eta_1(t)dt=1.$ Поэтому согласно
равенству (\ref{eq3}) и определению кольцевого $(p, Q)$-отображения,
см. (\ref{eqA2}),
\begin{equation}\label{eq5}
M_p\left(f\left(\Gamma\left(S_1,\,S_2,\,A\right)\right)\right)\ \
\leqslant\ \int\limits_{A} Q(x)\cdot {\eta_{1}}^p (|x-x_0|) \ dm(x)\
= \quad\frac{\omega_{n-1}}{I^{p-1}}\,,
\end{equation}
где $S_i=S(x_0, r_i).$ Пусть $\Gamma_E$ и $\Gamma_{f(E)}$ --
семейства кривых в смысле обозначений предложения \ref{pr1*!}, см.
также \cite[предложение~10.2 гл.~II]{Ri}. По этому предложению
\begin{equation}\label{eq6*!}
{\rm cap}_p\,f(E)\,=\,{\rm cap}_p\ \left(f\left(B\left(x_0,
r_2\right)\right), f\left(\overline{B(x_0
,r_1)}\right)\right)=M_p\left(\Gamma_{f(E)}\right).
\end{equation}
Пусть $\Gamma^{*}$ -- семейство максимальных поднятий
$\Gamma_{f(E)}$ с началом в $\overline{B\left(x_0, r_1\right)}.$
Будем иметь $\Gamma^{*}\subset \Gamma_E.$ Заметим, что
$\Gamma_{f(E)}>f(\Gamma^{*}),$ и что для достаточно малых
$\delta>0,$ $\Gamma_E> \Gamma\left(S(x_0, r_2-\delta), S(x_0, r_1),
D\right).$ Следовательно, в виду соотношения (\ref{eq5}) и свойству
минорирования, $ \Gamma_1\,>\,\Gamma_2 \quad\Rightarrow \quad
M_p(\Gamma_1)\leqslant M_p(\Gamma_2),$ см. \cite[теорема~6.4,
разд.~6]{Va}, получаем
$ M_p\left(\Gamma_{f(E)}\right)\leqslant
M_p\left(f(\Gamma^{*})\right)\leqslant
M_p\left(f(\Gamma_E)\right)\leqslant
$
$$\leqslant M_p\left(f\left(\Gamma\left(S(x_0, r_1), S(x_0,
r_2-\delta), A( r_1, r_2-\delta,
x_0)\right)\right)\right)\leqslant$$
\begin{equation}\label{eq6}
\leqslant\frac{\omega_{n-1}}{\left(\int\limits_{r_1}^{r_2-\delta}\frac{dt}{
t^{\frac{n-1}{p-1}}q_{x_0}^{\frac{1}{p-1}}(t)}\right)^{p-1}}\,.
\end{equation}
Заметим, что функция $\widetilde{\psi(t)}:=\psi|_{(r_1, r_2)} =
\frac{1}{t^{\frac{n-1}{p-1}}q_{x_0}^{\frac{1}{p-1}}(t)}$ суммируема
на $(r_1,r_2),$ поскольку по предположению $I\ne \infty.$ Отсюда, по
абсолютной непрерывности интеграла, получаем, что
\begin{equation}\label{eqA29}\int\limits_{r_1}^{r_2-\delta}\frac{dt}{t^{\frac{n-1}{p-1}}q_{x_0}
^{\frac{1}{p-1}}(t)}\rightarrow
\int\limits_{r_1}^{r_2}\frac{dt}{t^{\frac{n-1}{p-1}}q_{x_0}^{\frac{1}{p-1}}(t)}\end{equation}
при $\delta\rightarrow 0.$ Из (\ref{eq6}) и (\ref{eqA29}) следует,
что
\begin{equation}\label{eq7}
M_p(\Gamma_{f(E)})\leqslant
\frac{\omega_{n-1}}{\left(\int\limits_{r_1}^{r_2}\frac{dt}
{t^{\frac{n-1}{p-1}}q_{x_0}^{\frac{1}{p-1}}(t)}\right)^{p-1}}\,.
\end{equation}
Объединяя (\ref{eq6*!}) и (\ref{eq7}), получаем соотношение
(\ref{eq2A}).
\end{proof}$\Box$

\medskip
\begin{remark}\label{rem2A}
Для открытых дискретных кольцевых $(p, Q)$-отоб\-ра\-же\-ний
$f:D\rightarrow \overline{{\Bbb R}^n}$ интеграл в (\ref{eq9}) {\it
всегда конечен} для сколь угодно малых $r_1$ и произвольном $r_2>0,$
$r_2<{\rm dist\,}(x_0,
\partial D).$ В противном случае из (\ref{eq2A})
вытекает, что для некоторого  $\varepsilon_1>0,$ ${\rm cap}_p\,
f\left(\overline{B(x_0,\varepsilon_1)}\right)=0,$ откуда следует,
что множество $A:=f\left(\overline{B(x_0,\varepsilon_1)}\right)$
является всюду разрывным ввиду \cite[лемма~2.13]{MRV$_2$} (см. также
\cite[теорема~1.15, гл.~VII]{Ri}). Однако, ввиду открытости
отображения $f,$
$${\rm
Int\,}f\left(\overline{B(x_0,\varepsilon_1)}\right)\ne\varnothing\,,$$
что противоречит сделанному выше выводу о нульмерности множества
$A.$
\end{remark}

\medskip
В процессе получения основных результатов заметки мы также нуждаемся
в следующей лемме (её доказательство см., напр.,
\cite[лемма~2.2]{Sal}).

\medskip
\begin{lemma}\label{pr1}
Пусть  $x_0 \in {\Bbb R}^n,$ $0<r_1<r_2<{\rm dist}\,(x_0,
\partial D).$ Положим
\begin{equation}\label{eq8}
 \eta_0(r)=\frac{1}{Ir^{\frac{n-1}{p-1}}q_{x_0}^{\frac{1}{p-1}}(r)}\,,
\end{equation} где $I$ -- величина, определённая в (\ref{eq9}).  Тогда
\begin{equation}\label{eq10}
\frac{\omega_{n-1}}{I^{p-1}}\ =\ \int\limits_{A} Q(x)\cdot
\eta_0^p(|x-x_0|)\ dm(x)\ \leqslant\ \int\limits_{A} Q(x)\cdot
\eta^p(|x-x_0|)\ dm(x)
\end{equation}
для фиксированной измеримой функции $Q:{\Bbb R}^n\,\rightarrow
[0,\infty]$ такой, что $q_{x_0}(r)\ne \infty$ для п.в. $r>0,$ и
любой функции $\eta :(r_1,r_2)\rightarrow [0,\infty]$ такой, что
\begin{equation}\label{eqA20}
\int\limits_{r_1}^{r_2}\eta(r) dr = 1\,.
\end{equation}

В частности, из (\ref{eq10}) следует, что для кольцевых $(p,
Q)$-отображений в точке $x_0$ неравенство (\ref{eq2A}), вообще
говоря, не может быть улучшено.
\end{lemma}

\medskip
\begin{remark}\label{rem1}
Из леммы \ref{pr1} вытекает, что если отображение $f$ удовлетворяет
оценке (\ref{eq2A}), то отображение $f$ также удовлетворяет и
неравенству
\begin{equation}\label{eq4}
{\rm cap}_p\,f(E)\leqslant \int\limits_{A} Q(x)\cdot \eta^p(|x-x_0|)
dm(x)
\end{equation}
для произвольной измеримой по Лебегу функции $\eta:
(r_1,r_2)\rightarrow [0,\infty ]$ такой, что выполнено условие
(\ref{eqA2}), где $E$ -- конденсатор вида $E=\left(B(x_0, r_2),
\overline{B(x_0, r_1)}\right),$ $0<r_1<r_2< {\rm dist} \,
(x_0,\partial D),$ а $A$ -- сферическое кольцо с центром в точке
$x_0$ радиусов $r_1$ и $r_2.$

Предположим дополнительно, что $f:D\rightarrow \overline{{\Bbb
R}^n}$ -- гомеоморфизм. Тогда ввиду предложения \ref{pr1*!} величина
${\rm cap}_p\,E$ совпадает с левой частью в (\ref{eq3*!}). Таким
образом, получаем следующее утверждение: если $f$ -- гомеоморфизм,
удовлетворяющий условию (\ref{eq2A}), то $f$ является кольцевым $(p,
Q)$-отображением в области $D.$
\end{remark}

{\bf 3. Оценка меры образа шара.} В этом разделе получена  оценка
меры образа шара при открытых дискретных кольцевых $(p,
Q)$-отображениях. Впервые оценка площади образа круга при
квазиконформных отоб\-ра\-же\-ниях встречается в монографии
М.А.~Лаврентьева (\cite{La}). В.И.~Кругликовым была получена оценка
меры образа шара для ото\-бра\-же\-ний квазиконформных в среднем в
${\Bbb R}^n$ (см. \cite[лемма~9]{Kru}). В дальнейшем $q_0(r)$
обозначает $q_{x_0}(r)$ при $x_0=0.$ Справедлива следующая

\begin{lemma}\label{lem3.1} Пусть $n\geqslant 2,$ $p\geqslant 1,$ $Q:{\Bbb B}^n\rightarrow [0, \infty]$ --
заданная измеримая по Лебегу функция, $f:{\Bbb B}^n\rightarrow {\Bbb
B}^n$
 -- открытое дискретное кольцевое $(p,
Q)$\--отоб\-ра\-же\-ние. Предположим, что $q_0(r)\ne\infty$ для п.в.
$r\in\,(0,1)\,$ и что при некоторой постоянной $c>0$
\begin{equation}\label{cond}
\int\limits_{r<|x|<1} Q(x)\, {\psi}^{p}(|x|)\,dm(x)\leqslant c\cdot
J^{\alpha}(r)\,,
\end{equation}
где $\alpha \leqslant p$ и  ${\psi}(t)$ -- неотрицательная измеримая
(по Лебегу) функция на $(0,\infty)$ такая, что
$ 0<J(r):=\int\limits_{r}^1\, {\psi}(t)\, dt <\infty.$
Тогда при $1<p<n$ имеет место оценка
\begin{equation}\label{eqks*}m(f(B(0, r)))\leqslant\Omega_n\cdot
\left(1+\beta
J^{\frac{p-\alpha}{p-1}}(r)\right)^{-\frac{n(p-1)}{n-p}},\end{equation}
где
$\beta=\frac{n-p}{p-1}\left(\frac{\omega_{n-1}}{c}\right)^{\frac{1}{p-1}}\,,$
а при $p=n$
\begin{equation}\label{eqks**}m(f(B(0, r)))\leqslant\Omega_n\cdot\exp\left\{-\gamma
J^{\frac{n-\alpha}{n-1}}(r)\right\}\,,
\end{equation}
где $\gamma=n\left(\frac{\omega_{n-1}}{c}\right)^{\frac{1}{n-1}}.$
\end{lemma}

\begin{proof} Рассмотрим сферическое кольцо
$R_{t}=\{x\in{\Bbb B}^n:t<|x|<t+\triangle t\}.$ Пусть
$E:=(A_{t+\triangle t},C_t)$ -- конденсатор, где $C_t=\overline{B(0,
t)}$ и $A_{t+\triangle t}=B(0, t+\triangle t).$ Тогда
$f(E)=(f(A_{t+\triangle t}),f(C_t))$ -- конденсатор в ${\Bbb R}^n.$
В силу неравенства (\ref{eqks2.8}) получим
\begin{equation}\label{eqks1.9}{\rm cap}_{p}\left(f(A_{t+\triangle
t}), f(C_t)\right)\geqslant\frac{\left(\inf
m_{n-1}\,\sigma\right)^{p}}{m\left(f(A_{t+\Delta t})\setminus
f(C_t)\right)^{p-1}},\end{equation}
где $m_{n-1}\,\sigma$ -- $(n-1)$-мерная мера Лебега
$C^{\infty}$-многообразия $\sigma,$ являющегося границей
$\sigma=\partial U$ ограниченного открытого множества $U,$
содержащего $f(C_t)$ и содержащегося вместе со своим замыканием
$\overline{U}$ в $f(A_{t+\triangle t}),$ а точная нижняя грань
берется по всем таким $\sigma.$ С другой стороны, в силу леммы
\ref{lem1} имеем
\begin{equation} \label{eqks1.10}{\rm cap}_{p}\left(f(A_{t+\triangle
t}), f(C_t)\right)\leqslant \frac{\omega_{n-1}}{\left(
\int\limits_{t}^{t+\Delta t}\frac{ds}{s^{\frac{n-1}{p-1}}\,
q_0^{\frac{1}{p-1}}(s)}\right)^{p-1}}\,.
\end{equation}
Из неравенств (\ref{eqks1.9}) и (\ref{eqks1.10}) получим
$$\frac{\left(\inf m_{n-1}\, \sigma\right)^{p}}{m\left(f(A_{t+\Delta
t})\setminus f(C_t)\right)^{p-1}}\leqslant
\frac{\omega_{n-1}}{\left(\int\limits_{t}^{t+\Delta
t}\frac{ds}{s^{\frac{n-1}{p-1}}\,
q_0^{\frac{1}{p-1}}(s)}\right)^{p-1}}\,.$$
Применяя изопериметрическое неравенство $\inf
m_{n-1}\,\sigma\geqslant n\cdot\Omega_n^{\frac{1}{n}}
\left(m(f(C_t))\right)^{\frac{n-1}{n}},$
будем иметь
\begin{equation}\label{eqks1.11}n\cdot\Omega_n^{\frac{1}{n}}
\left(m(f(C_t))\right)^{\frac{n-1}{n}}\leqslant
\omega^{\frac{1}{p}}_{n-1} \left(\frac{m\left(f(A_{t+\Delta
t})\setminus f(C_t)\right)}{\int\limits_{t}^{t+\Delta
t}\frac{ds}{s^{\frac{n-1}{p-1}}\,
q_0^{\frac{1}{p-1}}(s)}}\right)^{\frac{p-1}{p}}.\end{equation}
Полагая $\Phi(t):=m\left(f(B(0, t))\right),$  из соотношения
(\ref{eqks1.11}) имеем, что
\begin{equation}\label{eqks1.12}n\cdot\Omega_n^{\frac{1}{n}}\Phi^{\frac{n-1}{n}}(t)\leqslant
\omega^{\frac{1}{p}}_{n-1} \left(\frac{\frac{\Phi(t+\Delta
t)-\Phi(t)}{\Delta t}}{\frac{1}{\Delta t}\int\limits_{t}^{t+\Delta
t}\frac{ds}{s^{\frac{n-1}{p-1}}\,
q_0^{\frac{1}{p-1}}(s)}}\right)^{\frac{p-1}{p}} \,.\end{equation}
Ввиду замечания \ref{rem2A},
$\frac{1}{s^{\frac{n-1}{p-1}}\, q_0^{\frac{1}{p-1}}(s)}\in
L^{1}_{loc} (0,1).$
Устремляя в неравенстве (\ref{eqks1.12}) $\Delta t$ к нулю, и
учитывая монотонное возрастание функции $\Phi(t)$ по $t\in(0,1)$ и
равенство $\omega_{n-1}=n\Omega_n,$ для п.в. $t$ имеем существование
производной $\Phi^{\,\prime}(t)$ и
\begin{equation}\label{eqks1.14}\frac{n\Omega^{\frac{p-n}{n(p-1)}}_n}
{t^{\frac{n-1}{p-1}}\,q_0^{\frac{1}{p-1}}(t)}\leqslant
\frac{\Phi^{\,\prime}(t)}{\Phi^{\frac{p(n-1)}{n(p-1)}}(t)}.\end{equation}

\medskip Рассмотрим неравенство (\ref{eqks1.14}) при $1<p<n.$
Интегрируя обе части неравенства по $t\in [r,1]$ и учитывая, что
$\int\limits_{r}^{1}\frac{\Phi^{\,\prime}(t)}{\Phi^{\frac{p(n-1)}{n(p-1)}}(t)}\,dt\leqslant
\frac{n(p-1)}{p-n}\left(\Phi^{\frac{p-n}{n(p-1)}}(1)-\Phi^{\frac{p-n}{n(p-1)}}(r)\right)$
(см., напр., \cite[теорема~7.4, гл.~IV]{Sa}), получим
\begin{equation}\label{eqks1.15}\Omega^{\frac{p-n}{n(p-1)}}_n
\int\limits_{r}^{1}\frac{dt}{t^{\frac{n-1}{p-1}}\,q_0^{\frac{1}{p-1}}(t)}\leqslant
\frac{p-1}{p-n}
\left(\Phi^{\frac{p-n}{n(p-1)}}(1)-\Phi^{\frac{p-n}{n(p-1)}}(r)\right).\end{equation}
Из неравенства (\ref{eqks1.15}) получаем, что
\begin{equation}\label{eq2.3.14098}\Phi(r)\leqslant\left(\Phi^{\frac{p-n}{n(p-1)}}(1)+\Omega^{\frac{p-n}{n(p-1)}}_n\,
\frac{n-p}{p-1}\int\limits_{r}^{1}\frac{dt}{t^{\frac{n-1}{p-1}}\,q_0^{\frac{1}{p-1}}(t)}\right)^
{\frac{n(p-1)}{p-n}}.
\end{equation}
В (\ref{eq10}) полагаем $\eta(t):=\frac{\psi(t)}{J},$ $t\in (r,
\,1),$ тогда по лемме \ref{pr1} и из неравенства (\ref{cond}) мы
получаем, что
\begin{equation}\label{eq2.3.14098po}
\frac{\omega_{n-1}}{\left(\int\limits_{r}^{1}\frac{dt}{t^{\frac{n-1}{p-1}}\,q_0^{\frac{1}{p-1}}(t)}\right)^{p-1}}\leqslant
\int\limits_{r<|x|<1} Q(x)\cdot \eta^{p}(|x|)\ dm(x)\leqslant  c
J^{\alpha-p}(r)\,.
\end{equation}
Наконец, комбинируя  (\ref{eq2.3.14098}) и (\ref{eq2.3.14098po}) и
учитывая, что $m(f({\Bbb B}^n))\leqslant\Omega_n,$ приходим к
(\ref{eqks*}).

Осталось рассмотреть случай $p=n.$ В этом случае неравенство
(\ref{eqks1.14}) примет вид:
\begin{equation}\label{eqks1.16}
\frac{n}{tq_0^{\frac{1}{n-1}}(t)}\leqslant\frac{\Phi^{\,\prime}(t)}{\Phi(t)}\,.
\end{equation}
Интегрируя обе части неравенства (\ref{eqks1.16}) по $t\in[r,1],$
учитывая, что
$\int\limits_{r}^{1}\frac{\Phi^{\,\prime}(t)}{\Phi(t)}dt\leqslant\ln\frac{\Phi(1)}{\Phi(r)}$
(см., напр., \cite[теорема~7.4, гл.~IV]{Sa}), получим
$n\int\limits_{r}^{1}\frac{dt}{tq_0^{\frac{1}{n-1}}(t)}\leqslant\ln\frac{\Phi(1)}{\Phi(r)}$
и, следовательно, имеем
$$\exp\left\{n\int\limits_{r}^{1}\frac{dt}{tq_0^{\frac{1}{n-1}}(t)}\right\}\leqslant
\frac{\Phi(1)}{\Phi(r)}\,,$$ а потому
\begin{equation}\label{eq2.3.14456}
\Phi(r)\leqslant\Phi(1)\cdot\exp\left\{-n
\int\limits_{r}^{1}\frac{dt}{tq_0^{\frac{1}{n-1}}(t)}\right\}\,.
\end{equation}
Аналогично, в (\ref{eq10}) полагаем $\eta(t):=\frac{\psi(t)}{J},$
$t\in (r, \,1),$ тогда по лемме \ref{pr1} имеем:
\begin{equation}\label{eq2.3.144567}
\frac{\omega_{n-1}}{\left(\int\limits_{r}^{1}\frac{dt}{tq_0^{\frac{1}{n-1}}(t)}\right)^{n-1}}\leqslant
\ \int\limits_{r<|x|<1} Q(x)\cdot \eta^{n}(|x|)\ dm(x)\leqslant
c\cdot J^{\alpha-n}(r)\,.
\end{equation}
Наконец, комбинируя  (\ref{eq2.3.14456}) и (\ref{eq2.3.144567}) и
учитывая, что $m(f({\Bbb B}^n))\leqslant\Omega_n,$ приходим к
(\ref{eqks**}).
\end{proof} $\Box$

{\bf 4. Поведение в точке. Аналог леммы Шварца.} Лемма, приведенная
в предыдущем разделе, позволяет  нам также описать асимптотическое
поведение открытых дискретных кольцевых $(p, Q)$-отображений в
начале координат. Всюду далее $D\ni 0:=(0,0,\dots, 0)$ и, как
прежде, $q_0(r)$ обозначает $q_{x_0}(r)$ при $x_0=0.$ Имеет место
следующее утверждение.

\medskip
\begin{theorem}\label{th2}
Пусть $n\geqslant 2,$ $p\geqslant 1,$ $Q:{\Bbb B}^n\rightarrow [0,
\infty]$ -- заданная измеримая по Лебегу функция, $f:{\Bbb
B}^n\rightarrow {\Bbb B}^n$ -- открытое дискретное кольцевое $(p,
Q)$-ото\-бра\-же\-ние, $f(0)=0.$ Если $q_0(r)\ne\infty$ для п.в.
$r\in (0, 1)$ и, кроме того, выполнено соотношение (\ref{cond}), где
$\alpha\leqslant p$ и ${\psi}(t)$ -- неотрицательная измеримая (по
Лебегу) функция на $(0,\infty)$, такая что
\begin{equation}\label{eq11}
0<J(r)=\int\limits_{r}^1\, {\psi}(t)\, dt <\infty\qquad \forall \,
r\in(0,1)\,,
\end{equation}
то при $1<p<n$ имеет место оценка
\begin{equation}\label{eqks1.21A}\liminf\limits_{x\rightarrow0}\,|f(x)|\cdot
\left(1+\beta
J^{\frac{p-\alpha}{p-1}}(|x|)\right)^{\frac{p-1}{n-p}}\leqslant1,\end{equation}
где
$\beta=\frac{n-p}{p-1}\left(\frac{\omega_{n-1}}{c}\right)^{\frac{1}{p-1}},$
а при $p=n$
\begin{equation}\label{eqks1.22}\liminf\limits_{x\rightarrow 0}\,|f(x)|\cdot
\exp\left\{\frac{\gamma}{n}J^{\frac{n-\alpha}{n-1}}(|x|)\right\}\leqslant
1
\end{equation}
и $\gamma=n\left(\frac{\omega_{n-1}}{c}\right)^{\frac{1}{n-1}}.$
\end{theorem}

\begin{proof}
Полагаем $\min\limits_{|x|=r}|f(x)|=l_f(r).$ Покажем, что
\begin{equation}\label{eqA27} B\left(0, l_f(r)\right)\subset
f\left(B(0, r)\right)
\end{equation}
при каждом $r\in (0, 1).$ Предположим противное. Тогда найдутся
$r_0\in (0, 1)$ и $y_0\in {\Bbb R}^n,$ такие, что $y_0\in B\left(0,
l_f(r_0)\right)$ и $y_0\in {\Bbb R}^n\setminus f\left(B\left(0,
r_0\right)\right),$ т.е., $y_0\in B\left(0, l_f(r_0)\right)\setminus
f\left(B\left(0, r_0\right)\right).$ Заметим, что $B\left(0,
l_f(r_0)\right)\cap f\left(B(0, r_0)\right)\ne \varnothing,$
поскольку соотношение $B\left(0, l_f(r_0)\right)\not\subset
f\left(B(0, r_0)\right),$ в частности, влечёт, что $l_f(r_0)>0$ и,
кроме того, условие $f(0)=0$ влечёт, что $0\in B\left(0,
l_f(r_0)\right)\cap f\left(B(0, r_0)\right).$ Заметим также, что шар
$B\left(0, l_f(r_0)\right)$ является связным множеством, при этом,
согласно сказанному выше, а также сделанному предположению,
$B\left(0, l_f(r_0)\right)\cap f\left(B(0,
r_0)\right)\ne\varnothing\ne B\left(0, l_f(r_0)\right)\setminus
f\left(B(0, r_0)\right).$ По \cite[теорема~1, гл.~5, \S\, 46]{Ku}
существует элемент $z_0\in B\left(0, l_f(r_0)\right)\cap
\partial f\left(B(0, r_0)\right).$ С другой стороны, согласно
свойству открытых отображений
$$\partial f\left(B(0, r_0)\right)\subset f\left(\partial \left(B(0,
r_0\right)\right)\,,$$
поэтому найдётся элемент $x_0\in S(0, r_0),$ такой что $f(x_0)=z_0.$
Однако, последнее невозможно, поскольку, в этом случае,
$f(x_0)=z_0\in B\left(0, l_f(r_0)\right),$ и, значит,
$|f(x_0)|<\min\limits_{|x|=r_0} |f(x)|$ при $x_0\in S(0, r_0).$
Полученное противоречие указывает на то, что предположение о
выполнении соотношения $B\left(0, l_f(r_0)\right)\not\subset
f\left(B(0, r_0)\right)$ было неверным и, значит, при всех $r\in (0,
1),$ справедливо включение (\ref{eqA27}).

Из соотношения (\ref{eqA27}), учитывая условие $f(0)=0,$ имеем
$\Omega_{n}\,l_{f}^{n}(r)\leqslant m\left(f\left(B(0,
r)\right)\right)$ и, следовательно,
\begin{equation}\label{eqks1.20}
l_{f}(r)\leqslant\left(\frac{m(f(B(0,
r))}{\Omega_n}\right)^{\frac{1}{n}}.\end{equation}
Таким образом, по лемме \ref{lem3.1} имеем:
$$\liminf\limits_{x\rightarrow 0}\frac{|f(x)|}{\mathcal{R}(|x|)}=\liminf\limits_{r\rightarrow 0}\frac{l_{f}(r)}{\mathcal{R}(r)}
\leqslant\liminf\limits_{r\rightarrow 0}\left(\frac{m(f(B(0,
r)))}{\Omega_n}\right)^{\frac{1}{n}}\cdot\frac{1}{\mathcal{R}(r)}\leqslant1\,,$$
где
$\mathcal{R}(r)=\left(1+\beta
J^{\frac{p-\alpha}{p-1}}(r)\right)^{-\frac{p-1}{n-p}}$ при $1<p<n$ и
$\mathcal{R}(r)=\exp\left\{-\frac{\gamma}{n}J^{\frac{n-\alpha}{n-1}}(r)\right\}$
при $p=n$ (где $J$ определено соотношением (\ref{eq11}), а
постоянные $\gamma$ и $\beta$ соответствуют формулировке леммы
\ref{lem3.1}). Теорема доказана.
\end{proof}$\Box$

\medskip
Из теоремы \ref{th2} вытекает следующая

\begin{theorem}\label{th3}
Пусть $n\geqslant 2,$ $p\geqslant 1,$ $Q:{\Bbb B}^n\rightarrow [0,
\infty]$ -- заданная измеримая по Лебегу функция, $f:{\Bbb
B}^n\rightarrow {\Bbb B}^n$ -- открытое дискретное кольцевое $(p,
Q)$-ото\-бра\-же\-ние, $f(0)=0.$ Предположим, что $q_0(r)\ne\infty$
для п.в.
 и
\begin{equation}\label{cond6xc}
\int\limits_{r<|x|<1} \frac{Q(x)}{|x|^p}\,dm(x)\leqslant c\ln
\left(\frac{1}{r}\right), \ \ \ \forall \, r\in(0,r_0)\,, r_0<1.
\end{equation}
 Тогда при $1<p<n$ имеет место оценка
\begin{equation}\label{eqks1.21AA}\liminf\limits_{x\rightarrow0}\,|f(x)|\cdot
\left(1+\beta \ln
\left(\frac{1}{|x|}\right)\right)^{\frac{p-1}{n-p}}\leqslant1,\end{equation}
где
$\beta=\frac{n-p}{p-1}\left(\frac{\omega_{n-1}}{c}\right)^{\frac{1}{p-1}}\,,$
а при $p=n$
\begin{equation}\label{eqks1.22A}\liminf\limits_{x\rightarrow 0}\,\frac{|f(x)|}{|x|^{\gamma_0}}\leqslant 1
\end{equation}
и $\gamma_0=\left(\frac{\omega_{n-1}}{c}\right)^{\frac{1}{n-1}}\,.$
\end{theorem}

\begin{proof}
Заключение теоремы \ref{th3} следует из теоремы \ref{th2} при
специальном выборе
$ \psi(t)\,=\,\left
\{\begin{array}{rr} \frac{1}{t}, &  \ t\in (0, 1) \\
0, & \ t\in \Bbb{R}\setminus (0, 1).
\end{array}\right.
$
\end{proof}$\Box$

\medskip
{\sc Доказательство теоремы \ref{th1}.} Рассмотрим  кольцо
$A=A(\varepsilon_1, \varepsilon_2, 0),$
$0<\varepsilon_1<\varepsilon_2<1.$  Пусть $E$ -- конденсатор вида
$E=\left(B(0, \varepsilon_2), \overline{B(0,
\varepsilon_1)}\right).$ Ввиду леммы \ref{lem1} и замечания
\ref{rem1} неравенство
\begin{equation}\label{eq12}
{\rm cap}_p\,f(E)\leqslant \int\limits_{A} Q(x)\cdot \eta^p(|x-x_0|)
dm(x)
\end{equation}
будет выполнено для произвольной измеримой по Лебегу функции $\eta:
(\varepsilon_1, \varepsilon_2)\rightarrow [0,\infty ]$ такой, что
$\int\limits_{\varepsilon_1}^{\varepsilon_2}\eta(r) dr\geqslant 1.$
Заметим, что функция
$$ \eta(t)\,=\,\left
\{\begin{array}{rr} \frac{1}{\varepsilon_2-\varepsilon_1}, &  \ t\in (\varepsilon_1,\varepsilon_2) \\
0, & \ t\in \Bbb{R}\setminus (\varepsilon_1,\varepsilon_2)
\end{array}\right.
$$
удовлетворяет последнему условию
$\int\limits_{\varepsilon_1}^{\varepsilon_2}\eta(r) dr\geqslant 1,$
поэтому, согласно (\ref{eq12}) мы получим, что
\begin{equation}\label{eq100} {\rm cap}_p\,\left(f(B(0, \varepsilon_2)),f(\overline{B(0, \varepsilon_1)})\right) \leqslant
\frac{1}{(\varepsilon_2-\varepsilon_1)^p}\int\limits_{A(\varepsilon_1,
\varepsilon_2, 0)} Q(x)\ dm(x)\ .\end{equation} Далее, выбирая
$\varepsilon_1=\varepsilon$ и $\varepsilon_2=2\varepsilon$, получим
\begin{equation}\label{eq101}{\rm cap}_p\,\left(f(B(0, 2\varepsilon)),f(\overline{B(0, \varepsilon)})\right)\leqslant\,
\frac{1}{\varepsilon^p}\int\limits_{A(\varepsilon, 2\varepsilon,
0)}Q(x)\,dm(x)\,.
\end{equation}
С другой стороны,  в силу  неравенства (\ref{maz}) мы имеем оценку:
\begin{equation}\label{eq102}
{\rm cap}_p\,\left(f(B(0, \varepsilon_2)),f(\overline{B(0,
\varepsilon_1)})\right)\geqslant c_1\left[m(f(B(0,
\varepsilon)))\right]^{\frac{n-p}{n}}\,,
\end{equation}
где  $c_1$  -- положительная постоянная, зависящая только от
размерности пространства  $n$ и заданного числа $p.$
Комбинируя   (\ref{eq101}) и (\ref{eq102}), получаем, что
\begin{equation}\label{eq4.2} \frac{m(f(B(0, \varepsilon))}
{\Omega_n \varepsilon^n }\leqslant c_{2} \
\left(\frac{1}{2^n\Omega_n\cdot\varepsilon^n}\int\limits_{B(0,
2\varepsilon)} Q(x)\, dm(x)\right)^{\frac{n}{n-p}} \,,\end{equation}
где $c_{2}$ - положительная постоянная зависящая только от $n$ и
$p.$

Положим $l_f(\varepsilon)=\min\limits_{|x|=\varepsilon}|f(x)|.$
Используя соотношение (\ref{eqks1.20}) мы получим, что
\begin{equation}\label{eqks1.200}
\liminf\limits_{x\rightarrow
0}\frac{|f(x)|}{|x|}=\liminf\limits_{\varepsilon\rightarrow
0}\frac{l_{f}(\varepsilon)}{\varepsilon}
\leqslant\liminf\limits_{\varepsilon\rightarrow
0}\left(\frac{m(f(B(0,
\varepsilon))}{\Omega_n\varepsilon^n}\right)^{\frac{1}{n}}\,.
\end{equation}
Наконец, комбинируя   (\ref{eq4.2}) и (\ref{eqks1.200}), имеем:
$$
\liminf\limits_{x\rightarrow 0}\frac{|f(x)|}{|x|}\leqslant c_{0} \
\liminf\limits_{\varepsilon\rightarrow
0}\left(\frac{1}{2^n\Omega_n\cdot\varepsilon^n}\int\limits_{B(0,
2\varepsilon)} Q(x)\, dm(x)\right)^{\frac{1}{n-p}}\,.
$$
Теорема  доказана. $\Box$

{\bf 5. Логарифмическая гёльдеровость открытых дискретных $(p,
Q)$-отоб\-ра\-же\-ний. Аналог теоремы Миньйович.} В начале статьи мы
уже обращали внимание на то, что в случае ограниченных $Q$ для
открытых дискретных $(p, Q)$-отоб\-ра\-же\-ний справедливы оценки
вида (\ref{eq13}). Как нами было обнаружено, некий аналогичный
результат, отличных от вышеприведенных, справедлив также и при
неограниченных $Q$ с <<допустимым ростом>>. (По этому поводу см.
также весьма существенный в этом направлении результат Р.~Миньйович
о необходимых и достаточных условиях равностепенной непрерывности
отображений с ограниченным искажением \cite{M}). Отметим, что в
случае неконформного модуля $p\in (n-1, n)$ равностепенная
непрерывность соответствующего класса отображений имеет место a
priori (см. \cite{GSS}), так что необходимые и достаточные условия
равностепенной непрерывности, в этом случае,  <<вырождаются>> в
необходимые условия, что находит своё отражение в следующем
утверждении.

\medskip
\begin{lemma}\label{lem6} 
Пусть $n\geqslant 2,$ $n-1<p<n,$ $Q:D\rightarrow [0, \infty]$ --
заданная измеримая по Лебегу функция, $f:D\rightarrow{\Bbb R}^n$
открытое дискретное кольцевое $(p, Q)$-отображение в точке $x_0\in
D.$ Предположим, найдутся числа $q < p,$ $\varepsilon_0 \in (0,\,
{\rm dist}\,(x_0,
\partial D)),$ $\varepsilon_0^{\,\prime}\in (0, \varepsilon_0)$ и неотрицательная измеримая
по Лебегу функция $\psi:(\varepsilon, \varepsilon_0)\rightarrow [0,
\infty],$ $\varepsilon\in\left(0, \varepsilon_0^{\,\prime}\right),$
такие что
\begin{equation} \label{eq3.7A}
\int\limits_{\varepsilon<|x-x_0|<\varepsilon_0}Q(x)\cdot\psi^p(|x-x_0|)
\ dm(x)\leqslant K\cdot I^q(\varepsilon, \varepsilon_0)\qquad
\forall\,\, \varepsilon\in(0,\varepsilon_0^{\,\prime})\,,
\end{equation}
где $I(\varepsilon, \varepsilon_0)$ определено соотношением
%
$0<I(\varepsilon, \varepsilon_0):=
\int\limits_{\varepsilon}^{\varepsilon_0}\psi(t)dt < \infty,$
$I(\varepsilon, \varepsilon_0)\rightarrow \infty$ при
$\varepsilon\rightarrow 0.$
%
Тогда найдутся постоянные $r_0>0$ и $N>0,$ зависящие только от
$\psi, n, x_0, K, p$ и $q$ такие, что для всех $x\in B(x_0, r_0)$
имеет место оценка
\begin{equation*}
|f(x)-f(x_0)|\leqslant N\cdot I^{\frac{(q-p)(n-1)}{p}}(|x-x_0|,
\varepsilon_0)\,.
\end{equation*}
\end{lemma}

\medskip
\begin{proof}
Воспользуемся неравенством (\ref{eq4}). Выберем в этом неравенстве в
качестве $\eta(t)=\psi(t)/I(\varepsilon, \varepsilon_0 ),$
$t\in(\varepsilon,\, \varepsilon_0).$ Тогда используя условие
(\ref{eq3.7A}) мы получим, что для конденсатора $E=\left(B(x_0,
\varepsilon_0), \overline{B(x_0, \varepsilon)}\right)$ справедливо
неравенство
\begin{equation}\label{eq3a}
{\rm cap}_p\,f(E)\leqslant K\cdot I^{q-p}\left(\varepsilon,
\varepsilon_0\right)\,.
\end{equation}
Поскольку по условию $I(\varepsilon, \varepsilon_0)\rightarrow
\infty$ при $\varepsilon\rightarrow 0,$ из соотношения (\ref{eq3a})
вытекает, что ${\rm cap\,}_p\,f(E)\leqslant \alpha(\varepsilon)$ для
всех $\varepsilon\,\in (0,\,\varepsilon^{\,\prime}_0),$ где
$\alpha(\varepsilon)\,\rightarrow\,0$ при $\varepsilon\,\rightarrow
0.$ Ввиду соотношения (\ref{maz}) мы получим, что
\begin{equation*}
\alpha(\varepsilon)\geqslant{\rm cap}_p\,f(E)\geqslant
n{\Omega}^{\frac{p}{n}}_n
\left(\frac{n-p}{p-1}\right)^{p-1}\left[m(f(C))\right]^{\frac{n-p}{n}}\,,
\end{equation*}
где $C:= \overline{B(x_0, \varepsilon)}.$ Из последнего соотношения
мы получим, что
$m(f(C))\leqslant \alpha_1(\varepsilon),$
где $\alpha_1(\varepsilon)\rightarrow 0$ при $\varepsilon\rightarrow
0.$ Учитывая последнее неравенство, выберем $\varepsilon_1>0,$
$\varepsilon_1<\varepsilon_0^{\,\prime},$ так что
\begin{equation}\label{eqroughb}
m(f(C_1))\leqslant 1\,,
\end{equation}
где $C_1:= \overline{B(x_0, \varepsilon_1)}.$ Рассмотрим теперь
конденсатор $\mathcal{E}=\left(B(0, \varepsilon_1), \overline{B(x_0,
\varepsilon)}\right),$ $\varepsilon<\varepsilon_1.$ Тогда снова
воспользуемся неравенством (\ref{eq4}), выбрав в этом неравенстве в
качестве $\eta(t)=\psi(t)/I(\varepsilon, \varepsilon_1 ),$
$t\in(\varepsilon,\, \varepsilon^{\,\prime}_1),$
$\varepsilon^{\,\prime}_1<\varepsilon_1.$ Тогда используя условие
(\ref{eq3.7A}) мы получим, что для конденсатора $\mathcal{E}$
справедливо неравенство
\begin{equation}\label{eq18}
{\rm cap\,}_p\,f\left({\mathcal E}\right)\leqslant
\alpha_2(\varepsilon)
\end{equation}
при всех $\varepsilon\in (0, \varepsilon^{\,\prime\prime}),$ где
$\alpha_2(\varepsilon)\rightarrow0$ при $\varepsilon\rightarrow 0.$
Воспользуемся теперь оценкой ёмкости конденсатора через меру и
диаметр:
\begin{equation}\label{eq16}
{\rm cap}_p\ E = {\rm cap}_p\ (A, C) \geqslant
\left(c_1\frac{\left(d(C)\right)^p}
{\left(m(A)\right)^{1-n+p}}\right)^{\frac{1}{n-1}}\,, \qquad p>n-1,
\end{equation}
где $c_1$ зависит только от $n$ и $p,$ а $d(C)$ обозначает диаметр
множества $C$ (см. \cite[предложение~6]{Kru}). Согласно этой оценке
и неравенству (\ref{eq18})
\begin{equation*}
\left(c_1\frac{\left(d(f(\overline{B(x_0, \varepsilon)}))\right)^p}
{\left(m(f(B(x_0,
\varepsilon_1)))\right)^{1-n+p}}\right)^{\frac{1}{n-1}} \leqslant
{\rm cap\,}_p\,f\left({\mathcal E}\right)\leqslant
\alpha_2(\varepsilon)\,.
\end{equation*}
Учитывая теперь неравенство (\ref{eqroughb}), из последнего
соотношения получаем, что
\begin{equation}\label{eq14}
d(f(\overline{B(x_0, \varepsilon)})) \leqslant
\alpha_3(\varepsilon)\,,
\end{equation}
где $\alpha_3(\varepsilon)\rightarrow0$ при $\varepsilon\rightarrow
0.$ Рассмотрим теперь отображение $\widetilde{f}(x):=f(x)-f(x_0).$
Тогда $d(f(\overline{B(x_0,
\varepsilon)}))=d(\widetilde{f}(\overline{B(x_0, \varepsilon)})),$
${\rm cap}_p\,f(E)={\rm cap}_p\,\widetilde{f}(E)$ и $\widetilde{f},$
как и $f,$ является открытым дискретным $(p, Q)$-отображением в
точке $x_0.$ В таком случае, из (\ref{eq14}) вытекает, что для
некоторого $r_1>0$ (зависящего только от $\psi, n, x_0, K, p$ и $q$)
выполнено
\begin{equation}\label{eq15}
|\widetilde{f}(x)|\leqslant 1\qquad \forall\,\,x\in B(x_0, r_1)\,.
\end{equation}
Рассмотрим теперь конденсатор $E=\left(B(0, r_2), \overline{B(x_0,
\varepsilon)}\right),$ где $0<\varepsilon<r_2<r_1.$ Полагая $A:=B(0,
r_2)$ и $C:=\overline{B(x_0, \varepsilon)},$ учитывая неравенство
(\ref{eq15}) (т.е., что $\widetilde{f}(A)\subset {\Bbb B}^n$),
применяя неравенство (\ref{eq16}) к отображению $\widetilde{f},$ мы
получим, что
\begin{equation}\label{eq17}
{\rm cap}_p\,f(E)\geqslant \left(c_1\frac{\left(d(f(C))\right)^p}
{\left(m(\widetilde{f}(A))\right)^{1-n+p}}\right)^{\frac{1}{n-1}}\geqslant
\left(c_1\frac{\left(d(f(C))\right)^p}
{\Omega_n^{1-n+p}}\right)^{\frac{1}{n-1}}\,.
\end{equation}
Из (\ref{eq17}) и (\ref{eq3a}) мы получим, что для некоторой
постоянной $M>0$ и при $\varepsilon< r_3$ (зависящего только от
$\psi, n, x_0, K, p$ и $q$)
\begin{equation}\label{eq20}
d\left(f(C)\right)\leqslant N\cdot
I^{\frac{(q-p)(n-1)}{p}}(\varepsilon, \varepsilon_0)\,.
\end{equation}
Зафиксируем теперь $x\in D$ так, что $|x-x_0|=\varepsilon,$ $0<
\varepsilon<r_3.$ Тогда
$x\in\overline{B\left(x_0,\,\varepsilon\right)},$
$f(x)\in\,f\left(\overline{B\left(x_0,\,\varepsilon\right)}\right)=f(C)$
и из (\ref{eq20}) вытекает, что для всех $x\in B(x_0, r_3)$
\begin{equation}\label{eq21}
|f(x)-f(x_0)|\leqslant N\cdot I^{\frac{(q-p)(n-1)}{p}}(|x-x_0|,
\varepsilon_0)\,.
\end{equation}
Лемма доказана.
\end{proof}

Следующее определение может быть найдено напр., в
\cite[разд.~6.1]{MRSY}. Будем говорить, что функция
${\varphi}:D\rightarrow{\Bbb R}$ имеет {\it конечное среднее
колебание} в точке $x_0\in D$, пишем $\varphi\in$ {\it FMO} в $x_0$,
если $\overline{\lim\limits_{\varepsilon\rightarrow 0}}\ \
\frac{1}{\Omega_n\varepsilon^n} \ \ \int\limits_{B(
x_0,\,\varepsilon)}
|{\varphi}(x)-\overline{{\varphi}}_{\varepsilon}|\ dm(x)\, <\,
\infty,
$
%
где
%
$\overline{{\varphi}}_{\varepsilon}\,=\,\frac{1}{\Omega_n\varepsilon^n}\int\limits_{B(
x_0,\,\varepsilon)} {\varphi}(x)\ dm(x).$ Приводимое ниже
утверждение более явно указывает на функции, при которых выполнено
утверждение леммы \ref{lem6}.

\medskip
\begin{proposition}\label{pr4.1*}
{\sl\, Пусть $Q:D\rightarrow [0,\infty]$ -- измеримая по Лебегу
функция, $D\subset{\Bbb R}^n,$ $n\geqslant 2,$ $x_0\in D,$ такая,
что либо $Q\in FMO(x_0),$ либо $I(x_0,\varepsilon, \varepsilon_1)=
\int\limits_{\varepsilon}^{\varepsilon_1}\
\frac{dr}{r^{\frac{n-1}{p-1}}q_{x_0}^{\frac{1}{p-1}}(r)}\rightarrow
\infty$ при $\varepsilon\rightarrow 0$ и некотором
$\varepsilon_1<{\rm dist}\,(x_0, \partial D)$ (где
$I(x_0,\varepsilon, \varepsilon_1)<\infty$ при каждом фиксированном
$\varepsilon\in (0, \varepsilon_1).$ Тогда можно указать
$\varepsilon_0\in (0,1)$ и функцию $\psi(t)>0$ такие, что в точке
$x_0$ выполнено условие (\ref{eq3.7A}) леммы \ref{lem6}.}
\end{proposition}

\begin{proof}
Пусть $Q\in FMO(x_0);$ не ограничивая общности, можно считать, что
$x_0=0.$ Пусть $\varepsilon_0<\min\left\{\,\,{\rm dist\,}\left(0,
\,\partial D\right),\,\, e^{\,-1}\right\}.$ На основании
\cite[следствие~6.3]{MRSY} для функции $0\,<\,\psi(t)\,=\,\frac
{1}{\left(t\,\log{\frac1t}\right)^{n/p}}$ будем иметь, что %
$\int\limits_{\varepsilon<|x|<\varepsilon_0} Q(x)\cdot\psi^p(|x|)
 \ dm(x)\,= \int\limits_{\varepsilon<|x|< {\varepsilon_0}}\frac{Q(x)\,
dm(x)} {\left(|x| \log \frac{1}{|x|}\right)^n} = O \left(\log\log
\frac{1}{\varepsilon}\right)$
%
при  $\varepsilon \rightarrow 0.$
Заметим также, что при указанных выше $\varepsilon$ выполнено
$\psi(t)\geqslant \frac {1}{t\,\log{\frac1t}},$ поэтому
$I(\varepsilon,
\varepsilon_0)\,:=\,\int\limits_{\varepsilon}^{\varepsilon_0}\psi(t)\,dt\,\geqslant
\log{\frac{\log{\frac{1}
{\varepsilon}}}{\log{\frac{1}{\varepsilon_0}}}}.$ Тогда соотношение
(\ref{eq3.7A}) выполнено при $q=1$ и, таким образом, утверждение
предложения \ref{pr4.1*} в случае $Q\in FMO$ доказано.

Пусть теперь $I(x_0,\varepsilon, \varepsilon_1)=
\int\limits_{\varepsilon}^{\varepsilon_1}\
\frac{dr}{r^{\frac{n-1}{p-1}}q_{x_0}^{\frac{1}{p-1}}(r)}\rightarrow
\infty$ при $\varepsilon\rightarrow 0$ и некотором
$\varepsilon_1<{\rm dist}\,(x_0, \partial D)$ (где
$I(x_0,\varepsilon, \varepsilon_1)<\infty$ при каждом фиксированном
$\varepsilon\in (0, \varepsilon_1).$ Положим
\begin{equation*}   
\psi(t)= \left \{\begin{array}{rr}
1/[t^{\frac{n-1}{p-1}}q_{x_0}^{\frac{1}{p-1}}(t)]\ , & \ t\in
(\varepsilon,\varepsilon_1)\ ,
\\ 0\ ,  &  \ t\notin (\varepsilon,\varepsilon_1)\ ,
\end{array} \right.
\end{equation*}
Тогда ввиду соотношения (\ref{eq3}) неравенство (\ref{eq3.7A}) также
выполнено при $q=1.$  Предложение \ref{pr4.1*} полностью доказано.
\end{proof}$\Box$

\medskip
Из леммы \ref{lem6} и предложения \ref{pr4.1*} (с учётом замечания
\ref{rem2A} в случае теоремы \ref{th5}) немедленно вытекают
следующие утверждения.

\medskip
\begin{theorem}\label{th4} 
Пусть $n\geqslant 2,$ $n-1<p<n,$ $Q:D\rightarrow [0, \infty]$ --
заданная измеримая по Лебегу функция, $f:D\rightarrow{\Bbb R}^n$
открытое дискретное кольцевое $(p, Q)$-отображение в точке $x_0\in
D.$ Предположим, что $Q\in FMO(x_0),$
тогда найдётся $r_0>0$ и постоянная $M>0,$ зависящие только от $n,
p, x_0$ и функции $Q$ такие, что для всех $x\in B(x_0, r_0)$ имеет
место оценка
\begin{equation}\label{eq19}
|f(x)-f(x_0)|\leqslant M\cdot
\left(\log\log(1/|x-x_0|)\right)^{\frac{(1-p)(n-1)}{p}}\,.
\end{equation}
В частности, соотношение (\ref{eq19}) выполнено, если функция $Q$
просто ограничена. Кроме того, из оценки (\ref{eq19}) вытекает
равностепенная непрерывность семейства всех указанных отображений
$f.$
\end{theorem}

\begin{theorem}\label{th5} 
Пусть $n\geqslant 2,$ $n-1<p<n,$ $Q:D\rightarrow [0, \infty]$ --
заданная измеримая по Лебегу функция, $f:D\rightarrow{\Bbb R}^n$
открытое дискретное кольцевое $(p, Q)$-отображение в точке $x_0\in
D.$ Предположим, что $I(x_0,\varepsilon, \varepsilon_1)=
\int\limits_{\varepsilon}^{\varepsilon_1}\
\frac{dr}{r^{\frac{n-1}{p-1}}q_{x_0}^{\frac{1}{p-1}}(r)}\rightarrow
\infty$ при $\varepsilon\rightarrow 0$ и некотором
$\varepsilon_1<{\rm dist}\,(x_0, \partial D).$
Тогда найдётся $r_0>0$ и постоянная $M>0,$ зависящие только от $n,
p, x_0$ и функции $Q$ такие, что для всех $x\in B(x_0, r_0)$ имеет
место оценка
\begin{equation}\label{eq20A}
|f(x)-f(x_0)|\leqslant
M\cdot\left(\int\limits_{|x-x_0|}^{\varepsilon_1}\
\frac{dr}{r^{\frac{n-1}{p-1}}q_{x_0}^{\frac{1}{p-1}}(r)}\right)^{\frac{(1-p)(n-1)}{p}}\,.
\end{equation}
Кроме того, из оценки (\ref{eq20A}) вытекает равностепенная
непрерывность семейства всех указанных отображений $f.$
\end{theorem}

\medskip
\noindent{{\bf Руслан Радикович Салимов} \\
Институт математики НАН Украины \\
ул. Терещенковская, д. 3 \\
г. Киев-4, Украина, 01 601\\
тел. +38 095 630 85 92 (моб.), e-mail: ruslan623@yandex.ru}

\medskip
\noindent{{\bf Евгений Александрович Севостьянов} \\
Житомирский государственный университет им.\ И.~Франко\\
ул. Большая Бердичевская, 40 \\
г.~Житомир, Украина, 10 008 \\ тел. +38 066 959 50 34 (моб.),
e-mail: esevostyanov2009@mail.ru}

\end{document}